\documentclass[10pt,a4paper,english,oneside]{amsart}
\usepackage{graphicx}												\usepackage{hyperref}
\usepackage{amssymb}
\usepackage[all,cmtip]{xy}
\date{\today}
\author{$\mbox{Mattia Ornaghi and Laura Pertusi}$}
\address{Dipartimento di Matematica "F.\ Enriques", Universit\`a degli Studi di Milano, Via Cesare Saldini 50, Milano 20133, Italy}
\email{mattia.ornaghi@unimi.it, https://sites.google.com/view/mattiaornaghi}
\email{laura.pertusi@unimi.it, http://www.mat.unimi.it/users/pertusi/}
\usepackage[english]{babel}
\usepackage{textcomp}
\usepackage[T1]{fontenc}
\usepackage{marvosym}
\usepackage{skull}
\usepackage{warning}
\usepackage{amstext}
\usepackage{etoolbox}
\usepackage{stmaryrd}
\pretocmd\mvchr{\text}{}{\errmessage{Patching \noexpand\mvchr failed}}
\pretocmd\textmvs{\text}{}{\errmessage{Patching \noexpand\textmvs failed}}

\usepackage{latexsym}
\usepackage{quoting}
\usepackage{guit}
\entrymodifiers={+!!<0pt,\fontdimen22\textfont2>}

\title{Voevodsky's conjecture for cubic fourfolds and Gushel-Mukai fourfolds via noncommutative K3 surfaces}
\usepackage{tikz}
\usetikzlibrary{backgrounds,calc,shadings,shapes.arrows,shapes.symbols,shadows}
\usepackage{amsmath, amsfonts, amsthm, amscd}
\usepackage{mathrsfs}
\usepackage[all,cmtip]{xy}	
\xyoption{arrow}
\usepackage{dsfont}														 
\usetikzlibrary{svg.path}

\DeclareMathOperator{\p}{\mathbb{P}}
\DeclareMathOperator{\G}{\text{Gr}}

\begin{document}

\theoremstyle{plain}
\newtheorem{thm}{Theorem}[section]
\newtheorem*{namedthm}{Theorem}%[section]
\newtheorem*{A}{Theorem}%[section]
\newtheorem*{B}{Theorem}
\newtheorem*{C}{Theorem}
\newtheorem{lem}[thm]{Lemma}
\newtheorem{prp}[thm]{Proposition}
\newtheorem{cor}[thm]{Corollary}
\newtheorem*{KL}{Klein's Lemma}
\theoremstyle{definition}
\newtheorem{defn}{Definition}[section]
\newtheorem*{conj}{Conjecture}%[section]
\newtheorem{exmp}{Example}[section]
\newtheorem{exr}{Exercize}[section]
\newtheorem*{soa}{State of art}
\theoremstyle{remark}
\newtheorem{rem}{Remark}
\newtheorem{note}{Note}
\newtheorem{case}{Case}
\newtheorem{oss}{Osservazione}
\newtheorem{question}{Question}
\newtheorem{dig}{Digression}[section]

\maketitle

\newcommand\blfootnote[1]{%
  \begingroup
  \renewcommand\thefootnote{}\footnote{#1}%
  \addtocounter{footnote}{-1}%
  \endgroup
}

\blfootnote{\textup{2000} \textit{Mathematics Subject Classification}: \textup{14A22}, \textup{14C15}, \textup{14J35}}

\begin{abstract}
In the first part of this paper we will prove the Voevodsky's nilpotence conjecture for smooth cubic fourfolds and ordinary generic Gushel-Mukai fourfolds.\ Then, making use of noncommutative motives, we will prove the Voevodsky's nilpotence conjecture for generic Gushel-Mukai fourfolds containing a $\tau$-plane $\G(2,3)$ and for ordinary Gushel-Mukai fourfolds containing a quintic del Pezzo surface.
\end{abstract}

%\tableofcontents

\section*{Introduction and statement of the results}

In 1995 Voevodsky conjectured the following statement for the algebraic cycles of a smooth projective $k$-scheme $X$: 
\begin{conj}[V]
$\mathcal{Z}^*_{\otimes_{\mbox{\tiny nil}}}(X)_F$ coincides with $\mathcal{Z}^*_{\otimes_{\mbox{\tiny num}}}(X)_F$.
\end{conj}
Here, $\mathcal{Z}^*(X)_F$ denotes the group of algebraic cycles of $X$, $\otimes_{\mbox{\tiny nil}}$ denotes the smash-nilpotence equivalence relation on  $\mathcal{Z}^*(X)_F$, introduced in \cite{Voe}, and $\otimes_{\mbox{\tiny num}}$ denotes the classical numerical equivalence relation on $\mathcal{Z}^*(X)_F$ (see $\S1$).\\
Voevodsky's nilpotence conjecture was proven for curves, surfaces, abelian threefolds, uniruled threefolds, quadric fibrations, intersection of quadrics, linear sections of Grassmannians, linear sections of determinantal varieties, homological projective duals and Moishenzon manifolds.\\
%Moreover, we recall that Voevodsky's nilpotence conjecture implies Kimura's conjecture and Schur's conjecture; consult $\S3$ for the precise statements.\\

In 2014, motivated by the above conjecture, Bernardara, Marcolli and Tabuada stated the following conjecture for a smooth and proper dg category $\mathcal{A}$:
\begin{conj}[$\mbox{V}_{\tiny\mbox{nc}}$]
$\mbox{K}_0(\mathcal{A})/_{\sim\otimes_{\tiny\mbox{nil}}}$ is equal to $\mbox{K}_0(\mathcal{A})/_{\sim\otimes_{\tiny\mbox{num}}}$.
\end{conj}

Here, $K_0(\mathcal{A})$ denotes the Grothendieck group of the full subcategory $\mathcal{D}_{\mbox{c}}(\mathcal{A})$ of compact objects of the derived category of $\mathcal{A}$, $\sim\otimes_{\tiny\mbox{nil}}$ and $\sim\otimes_{\tiny\mbox{num}}$ denote two equivalence relations on the Grothendieck group $K_0(\mathcal{A})$, as explained in $\S4$.\\

They also reformulate Voevodsky's conjecture in the following way:
\begin{namedthm}[{BMT}]\label{BMT1}
Let $X$ be a smooth projective $k$-scheme. The conjecture \emph{$\mbox{V}(X)$} is equivalent to the conjecture \emph{{$\mbox{V}_{\tiny\mbox{nc}}(\mbox{perf}_{\tiny\mbox{dg}}(X))$}}.
\end{namedthm}

Here, $\mbox{perf}_{\tiny\mbox{dg}}(X)$ denotes the unique enhancement of the derived category of perfect complexes on $X$. Making use of noncommutative motives, Voevodsky's conjecture was proven for example for quadric fibrations, intersection of quadrics, linear sections of Grassmannians, etc.\ (see \cite{BMT} and \cite{BeTa} for details).\\
%Roughly speaking they found a relations between the derived category and motivic decomposition ?? \\

The first result of this paper is the proof of Voevodsky's conjecture for cubic fourfolds and ordinary generic Gushel-Mukai fourfolds. We recall that a cubic fourfold is a smooth complex hypersurface of degree 3 in $\mathbb{P}^5$, while a Gushel-Mukai fourfold is a smooth and transverse intersection of the form $\mbox{Cone}(\mbox{Gr$(2,V_5)$})\cap\mbox{Q}$, where $\mbox{Q}$ is a quadric hypersurface in $\mathbb{P}^8$.
\begin{namedthm}[A]
\label{A}
Let $X$ be a cubic fourfold or an ordinary generic Gushel-Mukai fourfold; then the conjecture $V(X)$ holds.
\end{namedthm}
In order to prove this conjecture, we use the decomposition in rational Chow motives of a flat morphism computed by Vial in \cite{Via}.\\

Then, making use of noncommutative motives, we prove the noncommutative version of Voevodsky's nilpotence conjecture for the Kuznetsov category of a cubic fourfold and the GM category of an ordinary generic Gushel-Mukai fourfold. Indeed, we recall from \cite{Kuz} that the derived category of a cubic fourfold $X$ has a semiorthogonal decomposition of the form $\mbox{perf}(X)=\langle\mathcal{A}_X,\mathcal{O},\mathcal{O}(H),\mathcal{O}(2H)\rangle$. Here, the line bundles $\mathcal{O}$, $\mathcal{O}(H)$ and $\mathcal{O}(2H)$ are exceptional objects and $\mathcal{A}_X$ is a noncommutative K3 surface in the sense of Kontsevich. Analogously, in \cite{KuPe} they proved that the derived category of a Gushel-Mukai fourfold $X$ admits a 
semiorthogonal decomposition with four exceptional objects and a non-trivial part $\mathcal{A}_X$; again, the subcategory $\mathcal{A}_X$ is a noncommutative K3 surface in the sense of Kontsevich (see $\S5$).\\ 

Let $X$ be a cubic fourfold or a Gushel-Mukai fourfold; we denote by $\mathcal{A}_X^{\tiny\mbox{dg}}$ the dg enhancement of the category $\mathcal{A}_X$ induced from $\mbox{perf}_{\tiny\mbox{dg}}(X)$. Our second result is the proof of conjecture $\mbox{V}_{\tiny\mbox{nc}}$ for $\mathcal{A}_X^{\tiny\mbox{dg}}$ as explained below. 
 
\begin{namedthm}[B]
\label{B}
Let $X$ be a cubic fourfold or an ordinary generic Gushel-Mukai fourfold; then \emph{$\mbox{V}_{\tiny\mbox{nc}}(\mathcal{A}_X^{\tiny\mbox{dg}})$} holds.
\end{namedthm} 

Roughly speaking, we prove the noncommutative version of Voevodsky's nilpotence conjecture for some noncommutative K3 surfaces.\\

An application of the part of Theorem (B) concerning cubic fourfolds is the following result, which states the Voevodsky's conjecture for generic Gushel-Mukai fourfolds containing a $\tau$-plane. 

\begin{namedthm}[C]
Let $X$ be a generic Gushel Mukai fourfold containing a plane $P$ of type $\emph{Gr}(2,3)$; then \emph{$V_{\tiny\mbox{nc}}(\mbox{perf}_{\tiny\mbox{dg}}(X))$} holds.
\end{namedthm}
We point out that the proof of Theorem (C) is based on the fact that the semiorthogonal decomposition of $\mbox{perf}(X)$ contains the category of Kuznetsov associated to a cubic fourfold, as showed in \cite{KuPe}.\\

To the best of the authors knowledge, Theorem ({A}) and Theorem ({C}) prove Voevodsky's nilpotence conjecture in new cases.\\
%Our hope is that Theorem  Open up new possibility to solve motivic conjecture making use of noncommutative motives.\\
We believe that Theorem (B) provides a new tool for the proof of Voevodsky's conjecture of a smooth projective $k$-scheme whose derived category of perfect complexes contains the noncommutative K3 surface $\mathcal{A}_X$.\\
%We also provide an alternative proof of Theorem ? via classical theory of motives.\\

The plan of the paper is as follows.\ In Section 1 we briefly survey some constructions and basic properties of pure motives.\ In Section 2.1 and Section 2.2 we recall some definitions and important properties of cubic fourfolds and Gushel- Mukai varieties.\ Section 2.3 is devoted to prove Voevodsky's conjecture for cubic fourfolds and generic Gushel-Mukai fourfolds. Section 3 treats the theory of noncommutative motives and the relation with the theory of pure motives. 
In Section 4 we give a brief review of the notion of Kuznetsov category (resp.\ GM category) associated to a cubic fourfold (resp.\ a Gushel-Mukai variety). Then, we prove the Voevodsky's nilpotence conjecture for the Kuznetsov category and for the GM category of a generic ordinary Gushel-Mukai fourfold.
In Section 5 we show some consequences of the results proved in the previous sections. In particular, we prove the Voevodsky's conjecture for Gushel-Mukai fourfolds containing some particular surfaces.\\ 
%The paper is\\ 
%Plan of the paper\\
%Background in pure motives\\
%Cubic fourfolds and GM varieties\\
%Voevodsky conjecture for cubic fourfolds\\
%Pure motives vs Nc motives\\
%Kuznetsov category\\
%V for GM varieties.\\

%\begin{conj}[V]
%Let $X$ be a smooth projective $k$-scheme, then $\mathcal{Z}^*_{\otimes_{\mbox{\tiny nil}}}(X)_F$ coincides with $\mathcal{Z}^*_{\otimes_{\mbox{\tiny num}}}(X)_F$.
%\end{conj}

%\subsection*{State of art} 
%\begin{exmp}
%We consider a cubic 3-fold $\mathcal{Q}$, we have a quadric fibration $q:\tilde{\mathcal{Q}}\to\mathbb{P}^2$, where $\tilde{\mathcal{Q}}$ denotes the blow-up of $\mathcal{Q}$ \cite{Lahoz, Macri, Stellari}.
%\end{exmp}
%\subsection*{Goal}
%The aim of these notes is to prove conjecture for a generic cubic fourfold and for Gushel-Mukai fourfolds containing a plane.\\
%The plan of the paper is the following...\\

\textbf{Acknowledgements:} We thank Paolo Stellari and Alexander Kuznetsov for many useful suggestions and for his comments that made us improve this short article.\\
The first author wants to thank Marcello Bernardara, Gon\c calo Tabuada and Luca Barbieri Viale for many useful and interesting discussions.\ He would also like to thank the Department of Mathematics of MIT for its warm hospitality and excellent working conditions.\ The first author was supported by the "National Group of Algebraic and Geometric Structures and their Applications" (GNSAGA-INdAM).\\
The second author want to thank Gian Pietro Pirola for answering a question concerning Bertini's Theorem.\\

\textbf{Notations and conventions:} 
The letter $k$ will stand for a field.\ The letter $F$ will denote a commutative ring.\ Given a smooth projective $k$-scheme $X$, we will write $\text{CH}_i(X)$ for the Chow group of $i$-dimensional cycles modulo rational equivalence.\ 
Throughout the article we will assume that all cubic fourfolds and Gushel-Mukai varieties are smooth.
We will denote by $\mbox{perf($X$)}$ the derived category of perfect complexes of $\mathcal{O}_X$-modules and by $\mbox{perf}_{\tiny\mbox{dg}}(X)$ the corresponding (unique) dg enhancement. Moreover, if there exists a semiorthogonal decomposition of the form $\mbox{perf}(X)=\langle\mathcal{T}_1,...,\mathcal{T}_l\rangle$, then $\mathcal{T}_i^{\tiny\mbox{dg}}$ is the dg enhancement induced from $\mbox{perf}_{\tiny\mbox{dg}}(X)$. If $\mathscr{C}$ is an essentially small category, then $\mbox{Iso}[\mathscr{C}]$ denotes the set of isomorphism classes of objects of $\mathscr{C}$.

\section{Background in pure motives}
%In this first section we give just an idea of the definition of morives.\\

In this first section we give some information about the theory of pure motives. In particular, we define the group of algebraic cycles and some adequate equivalence relations on it. Then, we give an idea of the construction of the category of Chow motives; cf.\ \cite{And} for such a construction. Finally, we recall some basic properties of rational motivic decomposition we will use in the next.\\

Let $k$ be a field. 

\begin{defn}[Group of algebraic cycles]
Let $X$ be a smooth projective $k$-scheme. We define the group of algebraic cycles $\mathcal{Z}^*(X)$ to be the direct sum $\bigoplus_{d\in\mathbb{N}}\mathcal{Z}^d(X)$, where $\mathcal{Z}^d(X)$ denotes the group $$\mathcal{Z}^d(X):=\Bigl\{V=\sum_i n_iV_i,\mbox{ }\mbox{ }\mbox{ }\parbox{20em}{s.t. $n_i\in\mathbb{Z}$ and $V_i$ is an irreducible reduced closed subscheme with $\mbox{codim}_X(V_i)=d$}\mbox{$\Bigr\}$}.$$
\end{defn}

\begin{rem}
If $F$ is a commutative ring, we set $\mathcal{Z}^*(X)_F=\mathcal{Z}^*(X)\otimes F$.
\end{rem} 

For any pair $\alpha$, $\beta\in\mathcal{Z}^*(X)_F$, we can define their intersection product $\alpha\cdot\beta$ by intersection theory.
%If $\alpha$ and $\beta$ intersecting transversely, the intersection product is simply the intersection $\alpha\cap\beta$. 
%In general, we need to replace $\alpha$ with another cycle $\alpha'$ which intersect transversely with $\beta$ in order to define the intersection product.\\
In order to define a ring structure on the group of algebraic cycles induced by the intersection product, it is necessary to quotient the group by an adequate equivalence relation. We give some examples of adequate relations.

\begin{exmp}(Rational equivalence)
We say that two algebraic cycles $\alpha$ and $\beta$ in $\mathcal{Z}^d(X)_F$ are rationally equivalent ($\alpha\sim_{\tiny\mbox{rat}}\beta$) if there exists an algebraic cycle $\gamma\in\mathcal{Z}^d(X\times\mathbb{P}^1)_F$, flat over $\mathbb{P}^1$, such that $i^{-1}_0\gamma-i^{-1}_{\infty}\gamma=\alpha-\beta$. The maps $i_0:X\times\mathcal{f}0\mathcal{g}\to X\times\mathbb{P}^1$ and $i_\infty:X\times\mathcal{f}\infty\mathcal{g}\to X\times\mathbb{P}^1$ are the respective inclusions. In the case of divisors, the condition above is equivalent to say that there exists a rational function $f$ on $X$ such that $\alpha-\beta=Z(f)$. 
We call \emph{Chow ring} the ring $\mathcal{Z}^*(X)_F/_{\sim_{\tiny\mbox{rat}}}$ and we denote it by Chow($X$).
\end{exmp}

\begin{exmp}[Smash-nilpotence equivalence]\label{exmp:nil}
We say that an algebraic cycle $\alpha\in\mathcal{Z}^*(X)_F$ is smash-nilpotent equivalent
to zero if there exists a positive integer $n$ such that $\alpha^{\otimes n}$ is equal to $0$ in Chow($X$). Two algebraic cycles $\alpha$, $\beta\in\mathcal{Z}^*(X)_F$ are smash-nilpotent equivalent ($\alpha\sim_{\otimes_{\tiny\mbox{nil}}}\beta$) if the algebraic cycle $\alpha-\beta$ is smash-nilpotent equivalent to zero.
\end{exmp}

\begin{exmp}[Numerical equivalence]\label{exmp:num}
We say that an algebraic cycle $\alpha\in\mathcal{Z}^*(X)_F$ is numerically trivial if for all $\gamma\in\mathcal{Z}^{n-d}(X)_F$, $\gamma\cdot\alpha=0\in\mathcal{Z}^n(X)_F$. Two cycles $\alpha$ and $\beta$ are numerically equivalent $(\alpha\sim_{\tiny\mbox{num}}\beta)$ if the algebraic cycle $\alpha-\beta$ is numerically trivial.
\end{exmp}

Roughly speaking we can define the category of Chow motives whose objects are the triples $(X,p,r)$, where $X$ is a smooth projective $k$-scheme, $p$ is an idempotent endomorphism and $r$ is an integer, and whose morphisms are obtained by the algebraic cycles. We denote by $\mbox{Chow($k$)}$ the category of Chow motives. For further details about the construction of $\mbox{Chow($k$)}$ and $\mbox{Chow($k$)}_F$ consult \cite{And}.\\
We have a contravariant symmetric monoidal functor
\begin{align*}
\mathfrak{h}:\mbox{SmProj$(k)$}^{\tiny\mbox{op}} &\to \mbox{Chow($k$)}\\
X&\mapsto \mathfrak{h}(X),
\end{align*}
where $\mbox{SmProj($k$)}$ denotes the category of smooth projective $k$-schemes.\\
It is well known that $\mbox{Chow($k$)}$ is an additive, idempotent complete and rigid symmetric monoidal category. We list some properties of the functor $\mathfrak{h}$.

\subsection{Projective space}\label{subsec:a} Let us denote by $\mathds{1}$ the $\otimes$-unit of the category $\mbox{Chow($k$)}$; we recall that $\mathfrak{h}(\mathbb{P}^1)=\mathds{1}\oplus\mathbb{L}$, where $\mathbb{L}$ denotes the Lefschetz motive.\\
In more general terms, for every positive integer $n$, we have the decomposition $\mathfrak{h}(\mathbb{P}^n)=\bigoplus_{i=0}^n\mathds{1}(-i)$, where $\mathds{1}(1)$ denotes the Tate motive (i.e. the inverse of $\mathbb{L}$, formally $\mathds{1}(-1)=\mathbb{L}$) and $-(i)$ denotes $-\otimes \mathds{1}(1)^{\otimes i}$.\\
\subsection{Blowups}\label{subsec:b} The functor $\mathfrak{h}$ is "well behaved" with respect to blowups. In detail, let $X$ be a smooth projective variety over a field $k$ and let $j:Y\hookrightarrow X$ be a smooth closed subvariety of codimension $r$. Then the blowup $\pi_Y:\mbox{Bl}_Y(X)\to X$ of $X$ in $Y$ induces an isomorphism of Chow motives $\mathfrak{h}(X)\oplus\bigoplus_{i=1}^{r-1}\mathfrak{h}(Y)(i)\to\mathfrak{h}(\mbox{Bl}_{Y}(X))$ (see \cite{Man}). As a consequence, if $\mbox{dim$Y$}\le 2$, then $\mbox{V}(\mbox{Bl}_{Y}(X))$ holds if and only if $\mbox{V}(X)$ holds.\\

\subsection{Flat morphisms}\label{subsec:c} Consider a flat morphism $f:X\to B$ in $\mbox{SmProj$(k)$}$, with $X$ and $B$ of dimension $d_X$ and $d_B$, respectively.\ We denote by $X_b$ the fiber of $f$ over a point $b$ in $B$ and let $\Omega$ be a universal domain containing $k$. Assume that $\mbox{CH}_l(X_b)=\mathbb{Q}$, for all $0\le l<\frac{d_X-d_B}{2}$ and for all points $b\in B(\Omega)$. Then we have a direct sum decomposition of the Chow motive of $X$ as $\mathfrak{h}(X)\simeq \bigoplus_{i=0}^{d_X-d_B}\mathfrak{h}(B)(i)\oplus (Z,r,\lfloor \frac{d_X-d_B+1}{2}\rfloor)$, where $Z$ is a smooth and projective variety of dimension
$$d_Z=
\begin{cases}
d_B-1, \mbox{ if $d_X-d_B$ is odd,}\\
d_B, \mbox{ if $d_X-d_B$ is even.}
\end{cases}$$
For a complete proof of this result, we refer to \cite{Via}, Theorem 4.2, Corollary 4.4. 

\begin{rem}
We point out that the same results hold for the category $\mbox{Chow}(k)_F$ for any field $F$.
\end{rem}

\section{Cubic fourfolds and Gushel-Mukai varieties}

In this section we briefly recall some facts about cubic fourfolds and Gushel-Mukai varieties.\ From now on the field $k$ is the complex field $\mathbb{C}$.\ We are interested in finding a quadric fibration from the blow-up of a cubic fourfold (resp.\ a GM fourfold) over a line (resp.\ a surface) to the projective space $\mathbb{P}^3$.  Then, we show how to use this geometric construction to prove Voevodsky's nilpotence conjecture for cubic fourfolds and ordinary generic GM fourfolds.

\subsection{Cubic fourfolds}
%Let us recall some facts about cubic fourfolds.
\begin{defn}[Cubic fourfold]
A \emph{cubic fourfold} is a smooth complex hypersurface of degree 3 in $\mathbb{P}^5$.
\end{defn}

We observe that a cubic fourfold $X$ contains (at least) a line $l$.\ Indeed, let $F(X)$ be the Fano variety of lines associated to $X$, which parametrizes the lines of $\p^5$ contained in $X$.\ By \cite{Fan}, \cite{AlKl}, \cite{BaVdV}, we know that $F(X)$ is a smooth and connected projective variety of dimension $4$.\ Actually, by \cite{BeDo}, we have that $F(X)$ is an irreducible holomorphic symplectic fourfold, deformation equivalent to the Hilbert scheme of points of length two on a K3 surface.\ In particular, $X$ contains at least a line, as we claimed.
  
We denote by $l$ a line in $X$ and let $\text{Bl}_l(X)$ be the blow-up of $X$ in $l$. The aim of this paragraph is to prove that the projection from the line $l$ induces a flat and smooth quadric fibration from $\text{Bl}_l(X)$.
%In general, let $X$ is a cubic hypersurface in $\mathbb{P}^{d+1}$, we can associate to $X$ a closed subvariety in the following way.

%\begin{defn}[Fano variety of lines]
%Let $X$ be a smooth projective variety in $\mathbb{P}^{d+1}$. We define the \emph{Fano variety of lines} associated to $X$ to be the variety, denoted by $F(X)$, given by $$F(X):=\mathcal{f}l\in\mbox{Gr$(1,\mathbb{P}^{d+1})$ such that $l\subset X$}\mathcal{g},$$ where $\mbox{Gr$(1,\mathbb{P}^{d+1})$}= \mbox{Gr}(2, d + 2)$ denotes the Grassmaniann of lines in $\p^{d+1}$.
%\end{defn}
%By \cite{Fan, AlKl, BaVdV}, we know that $F(X)\not=\emptyset$ if $d\ge2$ and $F(X)$ is connected if $d\ge3$; furthermore, $F(X)$ is a smooth and projective variety of dimension $2(d-2)$ when $X$ is smooth. Moreover, if $X$ is a cubic fourfold, then the associated Fano variety of lines $F(X)$ is an irreducible holomorphic symplectic fourfold of K3 type by \cite{BeDo}. In particular, $X$ contains certainly (at least) a line. 

%\begin{exmp}
%The equation $x^2_0x_1+x^2_1x_2+x^2_2x_0+x^2_3x_4+x^2_4x_5+x^2_5x_3$ defines a smooth cubic fourfold.
%\end{exmp}

\begin{lem}\label{IN}
Let $X$ be a smooth cubic fourfold and let $l$ be a line in $X$.\ Then the linear projection from the line $l$ induces a smooth flat quadric fibration from the blow-up $\emph{Bl}_l(X)$ to $\mathbb{P}^3$.
\end{lem}

\begin{proof}
Let $V_6$ be a six-dimensional vector space such that $X \subset \p(V_6) \cong \p^5$. Let $V_2$ be a two-dimensional subvector space of $V_6$ such that $l= \p(V_2)\cong \p^1$ and we set $V_4:=V_6/V_2$. We denote by $\text{Bl}_l(\p(V_6))$ the blow-up of $\p(V_6)$ in $l$. Then the projection from the line $l$ defines a regular map $\pi: \text{Bl}_l(\p(V_6)) \rightarrow \p(V_4)\cong \p^3$, which is a $\p^2$-bundle over $\p^3$.  Let $\pi_l: \text{Bl}_l(X)  \rightarrow X$ be the blow-up of $X$ along $l$. Then the restriction of $\pi$ to $\text{Bl}_l(X)$ induces a smooth flat conic fibration $f: \mbox{Bl}_{l}(X) \rightarrow \p(V_4)\cong \p^3$. In other words, we have the following commutative diagram
\begin{equation*}
\label{DC2}
\xymatrix{
\text{Bl}_l(X) \ar[d]_{\pi_l} \ar@{^{(}->}[r] \ar[rd]_f & \text{Bl}_l(\p(V_6)) \ar[d]^{\pi}\\
X \subset \p(V_6)   & \p(V_4) 
} 
\end{equation*}  
where $f$ is the map claimed in the statement. 
 
%We know that $X$ contains a line, we denote by $l$ such a line.
%Let $\pi_l$ be the projection from $l$ onto $\mathbb{P}^3$. We consider the blowup $\mbox{Bl}_l(X)$ of $X$ along $l$, then we have the following induced map:
%\[
%\xymatrix{
%\pi_{l}:X\ar@{-->}[r]& \mathbb{P}^3
%}\leadsto
%\xymatrix{
%f:\mbox{Bl}_{l}X\ar@{-->}[r]& \mathbb{P}^3.
%}
%\]
\end{proof}

\subsection{Gushel-Mukai varieties}
Let $V_5$ be a $k$-vector space of dimension 5; considering the Pl\"ucker embedding, we have that $\mbox{Cone}(\mbox{Gr$(2,V_5)$})\subset\mathbb{P}(k\oplus\wedge^2V_5)$. We denote by $W$ a linear subspace of dimension $n+5$ of $\wedge^2V_5\oplus k$ (with $2\le\mbox{$n$}\le 6$). 

\begin{defn}[Gushel-Mukai $n$-fold]
We define a \emph{Gushel-Mukai $n$-fold} $X$ to be a smooth and transverse intersection of the form $$X=\mbox{Cone}(\mbox{Gr$(2,V_5)$})\cap\mbox{Q},$$
where $\mbox{Q}$ is a quadric hypersurface in $\mathbb{P}(W)$.\\
We say that $X$ is: 
\begin{itemize}
\item \emph{Ordinary} if $X$ is isomorphic to a linear section of $\mbox{Gr}(2,V_5)\subset\mathbb{P}^9$,
\item \emph{Special} if $X$ is isomorphic to a double cover of a linear section of $\mbox{Gr}(2,V_5)$ branched along a quadric section.
\end{itemize}
From now on, we will write GM instead of Gushel-Mukai.
\end{defn}

Let $X$ be a GM variety. Notice that $X$ does not contain the vertex of the cone over $\mbox{Gr}(2,V_5)$, because $X$ is smooth. Thus, we have a regular map defined by the projection from the vertex:
$$\gamma_X:X\to\mbox{Gr(2,$V_5$)}.$$

\begin{defn}[Gushel bundle]
Let $\mathscr{U}$ be the tautological bundle of rank 2 over $\mbox{Gr$(2,V_5)$}$. We define the \emph{Gushel bundle} to be the pullback $\mathscr{U}_X:=\gamma_X^*\mathscr{U}$.
\end{defn}

We denote by $\pi:\p_X(\mathscr{U}_X) \rightarrow X$ the projectivization of the bundle $\mathscr{U}_X$. We can consider the map
$$\rho: \p_X(\mathscr{U}_X) \rightarrow \p(V_5)$$
induced by the embedding $\mathscr{U}_X \hookrightarrow V_5 \otimes \mathcal{O}_X$. By \cite{DeKu}, Proposition 4.5, we have that $\rho$ is a fibration in quadrics.

Now, let us suppose that $X$ is an ordinary GM fourfold. By \cite{DeKu}, Remark 3.15 and Remark B.4, the fibers of $\rho$ are all conics in $\p^2$ except for the fiber over a point $v_0$ in $\p(V_5)$, which is a $2$-dimensional quadric in $\p^3$. Let us fix a four-dimensional subvector space $V_4$ of $V_5$ such that the point $v_0$ is not contained in $\p(V_4)$. We set
$$\tilde{X}:= \p_X(\mathscr{U}_X) \times_{\p(V_5)} \p(V_4)$$
and we denote by $\tilde{\rho}$ the restriction of $\rho$ to $\tilde{X}$. Thus, we have the following commutative diagram
\begin{equation}
\label{DC1}
\xymatrix{
  & \tilde{X} \ar[dl]_{\sigma} \ar[d]_{\tilde{\rho}} \ar[r] & \p_X(\mathscr{U}_X) \ar[d]^{\rho} \ar[dll]^{\pi} \\
X  &\ar[l] \p(V_4)  \ar[r]                        & \ar@/^/[ll]  \p(V_5)} 
\end{equation}                         
By the previous observations, we have that the restriction $\tilde{\rho}$ defines a flat conic fibration over $\p(V_4) \cong \p^3$. In the rest of this section, we prove that $\tilde{\rho}$ is smooth when $X$ is generic.

Notice that for every $x$ in $X$, the fiber of $\sigma$ over $x$ is equal to $\p(\mathscr{U}_{X,x} \cap V_4)$. In particular, we have that $\sigma^{-1}(x)$ is a point (resp.\ a line) if the dimension of $\mathscr{U}_{X,x} \cap V_4$ is equal to $1$ (resp.\ if $\mathscr{U}_{X,x} \subset V_4$). It follows that the locus of non trivial fibers of $\sigma$ is the intersection 
\begin{equation}
\label{E}
E:= \G(2,V_4) \cap X=\G(2,V_4) \cap \p(W) \cap \mbox{Q} \subset \p(\bigwedge^2 V_5) \cong \p^9.
\end{equation} 
Since the Grassmannian $\G(2,V_4)$ has degree $2$, we have that the degree of $E$ is at most $4$. Moreover, the expected dimension of $E$ is $2$. On the other hand, by Lefschetz Theorem the fourfold $X$ cannot contain a divisor with degree less than $10$, because its class has to be cohomologous to the class of a hyperplane in $X$. Thus, we conclude that $\text{dim}(E) \leq 2$. In the next lemma, we show that $E$ is smooth under generality assumptions on $\p(W)$ and $\mbox{Q}$; in this case, $E$ is a del Pezzo surface of degree $4$.

\begin{lem}
\label{lemmanuovo}
If $W$ is a generic vector space of dimension $9$ in $\bigwedge^2V_5$ and $\mbox{Q}$ is a generic quadric hypersurface in the linear system $|\mathcal{O}_{\p(W)}(2)|$, then $E$ defined in \eqref{E} is a smooth and irreducible surface.
\end{lem}
\begin{proof}
We consider the intersection $Y:=\p(W) \cap \G(2,V_4) \subset \p(\bigwedge^2V_5) \cong \p^9$. By Bertini's Theorem on hyperplane sections (see \cite{GH}), we have that $Y$ is smooth and irreducible, because $\p(W)$ is a generic hyperplane in $\p^9$. 

Let $i:Y \hookrightarrow \p^8$ be the embedding of $Y$ in $\p(W) \cong \p^8$. Notice that if $Y$ is contained in the quadric $\mbox{Q}$, then $Y=E$ would be a smooth divisor in $X$ with degree less than $10$, in contradiction with the previous observation. Hence, we have that the quadric $\mbox{Q}$ does not contain $Y$. Again by Bertini's Theorem, the intersection $Y \cap \mbox{Q}=E$ is smooth and irreducible. Indeed, we can consider the embedding of $\p^8$ in $\p(H^0(\p^8,\mathcal{O}(2))) \cong \p^N$ defined by $\mathcal{O}(2)$. The quadric hypersurfaces in $\p^8$ correspond to hyperplanes in $\p^N$ via this embedding. Thus, by Bertini's Theorem for hyperplane sections, we conclude that the intersection of the image of $Y$ with the generic hyperplane in $\p^N$, corresponding to the generic quadric $\mbox{Q}$, is smooth and irreducible. Hence, we conclude that $E$ is smooth and irreducible of dimension $2$, as we wanted.       
\end{proof}

As a consequence, we obtain the smoothness of the restriction to a hyperplane of the conic fibration $\rho$.   

\begin{cor}
\label{fibrazione}
Let $X$ be an ordinary generic GM fourfold. Then $\tilde{X}$ is the blow-up of $X$ in $E$ and the map $\tilde{\rho}: \tilde{X} \rightarrow \p(V_4)$ defined in \eqref{DC1} is a smooth flat conic fibration.  
\end{cor}
\begin{proof}
We observe that the quadric $\mbox{Q}$ which defines $X$ is generic in the linear system $|\mathcal{O}_{\p(W)}(2)|$, because $X$ is a generic quadric section of the hull $\p(W) \cap \G(2,V_5)$. On the other hand, we recall that, by \cite{DeKu}, Lemma 2.7, there exists a functor between the groupoid of polarized GM varieties to the groupoid of GM data, which is an equivalence by \cite{DeKu}, Theorem 2.9. In particular, a generic $X$ corresponds to a generic GM data $(W, V_6,V_5,L,\mu,\textbf{q},\varepsilon)$. Thus, the vector spaces and the linear maps which define this GM data are generic and, then, $W$ is a generic subvector space in $\bigwedge^2 V_5$. By Lemma \ref{lemmanuovo}, we have that the locus $E$ defined by \eqref{E} is smooth and irreducible. 

Notice that $\sigma^{-1}(E)$ is by definition the projective bundle $\p_E(\mathcal{U}_X) \rightarrow E$. On the other hand, the exceptional divisor of the blow-up of $X$ in $E$ is isomorphic to the projectivized conormal bundle $\p_E(\mathcal{N}_{E|X}^*)$. Since $E$ can be represented as the zero locus of a regular section of $\mathscr{U}_X^*$, the conormal bundle of $E$ in $X$ is isomorphic to $\mathscr{U}_X$. Hence, we deduce that $\tilde{X}$ is the blow-up of $X$ in $E$. It follows that $\tilde{X}$ is smooth and $\tilde{\rho}: \tilde{X} \rightarrow \p(V_4)$ is a smooth flat conic fibration, as we claimed. 
\end{proof} 

\subsection{Proof of Voevodsky's nilpotence conjecture for cubic fourfolds and generic GM fourfolds}

In \cite{Voe} Voevodsky conjectured the following statement for the algebraic cycles: 

\begin{conj}[V]
Let $X$ be a smooth projective $k$-scheme; let $\mathcal{Z}^*_{\otimes_{\mbox{\tiny nil}}}(X)_F$ and $\mathcal{Z}^*_{\otimes_{\mbox{\tiny num}}}(X)_F$ be the ring of algebraic  cycles modulo the relation of Example \ref{exmp:nil} and of Example $\ref{exmp:num}$, respectively. Then $\mathcal{Z}^*_{\otimes_{\mbox{\tiny nil}}}(X)_F$ coincides with $\mathcal{Z}^*_{\otimes_{\mbox{\tiny num}}}(X)_F$.
\end{conj}
\begin{soa}
Conjecture V was proven for curves, surfaces, abelian threefolds, uniruled threefolds, quadric fibrations, intersection of quadrics, linear sections of Grassmannians, linear sections of determinantal varieties, homological projective duals\ %and Moishenzon manifolds 
(see \cite{And}, \cite{KaSe}, \cite{Mat}, \cite{Voe}, \cite{Voi}, \cite{BMT}).
\end{soa}

\begin{A}[A]\label{A}
Let $X$ be a cubic fourfold or an ordinary generic $GM$ fourfold. Then the conjecture $V(X)$ holds.
\end{A}

\begin{proof}
Let $X$ be a cubic fourfold and we consider the blow-up of $X$ along a line $l$. By Subsection \ref{subsec:c} and Lemma \ref{IN}, the Chow motive of the blow-up decomposes as 
$$\mathfrak{h}(\mbox{Bl}_l(X))\simeq\bigoplus_{k=0}^{1}\mathfrak{h}(\mathbb{P}^3)(k)\oplus(Z,r,1)\simeq\bigoplus_{k=0}^{1}(\bigoplus^{3}_{i=0}\mathds{1}(-i))(k)\oplus(Z,r,1),$$ where $r\in\mbox{End$(\mathfrak{h}(Z))$}$ and $\mbox{dim $Z$}=\mbox{dim $\mathbb{P}^3-1$}=2$. It means that conjecture V holds for $\mbox{Bl}_l(X)$: by Subsection \ref{subsec:b} we conclude that conjecture V holds for $X$, as we wanted.

If $X$ is an ordinary generic GM fourfold, the same strategy applied to the conic fibration of Corollary \ref{fibrazione} gives the required statement. 
\end{proof}

%As a consequence of Theorem \ref{A} and Theorem \ref{CC}, we obtain the following result. 
%\begin{cor}
%Let $X$ be a cubic fourfold or an ordinary generic GM fourfold. Then the conjectures $K(X)$ and $S(X)$ hold.
%\end{cor}

\section{Pure motives vs noncommutative motives}

The aim of this section is to give a comparison between the theory of pure motives and the theory of noncommutative motives introduced in \cite{Tab}.
This section is divided in five parts: in the first and in the second we recall some basic properties of dg categories and dg enhancements. Then we give the notion of noncommutative Chow motives and we state Voevodsky's nilpotence conjecture in the noncommutative case. Finally, we relate such a conjecture with the classical Voevodsky's nilpotence conjecture for pure motives.

\subsection{Dg categories}

Let $\mathcal{C}(k)$ be the category of differential graded $k$-modules.
We define a differential graded category (shortly a dg-category) to be a category enriched over $\mathcal{C}(k)$ (morphism sets are complexes) whose compositions fulfill the Leibniz rule: 
$$d(f\circ g)=d(f)\circ g + (-1)^{\tiny\mbox{deg($f$)}}f\circ d(g).$$
A dg functor is a functor between dg categories enriched over $\mathcal{C}(k)$. We denote by dgcat the category whose objects are small dg categories and whose morphisms are dg functors. Consult \cite{Kel} for a precise reference.\\
Let $\mathcal{A}$ be a dg category. We define the opposite category $\mathcal{A}^{\tiny\mbox{op}}$ to be the dg category with the same objects of $\mathcal{A}$ and whose morphisms are given by $\mathcal{A}^{\tiny\mbox{op}}(x, y):=\mathcal{A}(y, x)$. A right dg $\mathcal{A}$-module is a dg functor $M:\mathcal{A}^{\tiny\mbox{op}}\to\mathcal{C}^{\tiny\mbox{dg}}(k)$, where $\mathcal{C}^{\tiny\mbox{dg}}(k)$ is the dg category of dg $k$-modules. We denote by $\mbox{Mod-}\mathcal{A}$ the category of dg $\mathcal{A}$-modules. The derived category of $\mathcal{A}$, denoted by $\mathcal{D}(\mathcal{A})$, is the localization of $\mbox{Mod-}\mathcal{A}$ with respect to the class of objectwise quasi-isomorphisms.\\
A dg functor $F:\mathcal{A}\to\mathcal{B}$ is a Morita equivalence if the induced functor $\mathbb{L}F_{!}:\mathcal{D}(\mathcal{A})\to\mathcal{D}(\mathcal{B})$ on derived categories is an equivalence of triangulated categories.\\
We note that the tensor product of $k$-algebras gives rise to a symmetric monoidal structure $\mbox{-}\otimes\mbox{-}$ on dgcat. The $\otimes$-unit is the dg category with one object $k$.\\ 
Moreover we say that a dg category $\mathcal{A}$ is smooth if it is perfect as a bimodule over itself. We say that $\mathcal{A}$ is proper if for every couple of objects $x, y\in\mathcal{A}$ the complex of $k$-modules $\mathcal{A}(x, y)$ is perfect. The definitions of smooth and proper dg category are due to Kontsevich.

\subsection{Dg enhancements}

Let $X$ be a smooth projective $k$-scheme. We know that the category of perfect complexes $\mbox{perf}(X)$ has a unique dg enhancement $\mbox{perf}_{\tiny\mbox{dg}}(X)$ (cf. \cite{LuOr} or \cite{CaSt}), which is smooth and proper as a dg category.\\
Moreover, suppose that the derived category of perfect complexes on $X$ has a semiorthogonal decomposition of the form $\mbox{perf}(X)=\langle\mathcal{A}_1,...,\mathcal{A}_{n}\rangle$. Then, by \cite{BMT}, we have that every dg category $\mathcal{A}^{\tiny\mbox{dg}}_i$ is smooth and proper (where $\mathcal{A}^{\tiny\mbox{dg}}_i$ denotes the dg enhancement of the subcategory $\mathcal{A}_i$ induced from $\mbox{perf}_{\tiny\mbox{dg}}(X)$).

\subsection{Noncommutative Chow motives}
We briefly recall the construction of noncommutative Chow motives; for the complete explanation see \cite{Tab}. First of all, we denote by Hmo the localization of dgcat with respect to the class of Morita equivalences. We observe that the tensor product of dg categories gives rise to a symmetric monoidal structure on Hmo($k$).\\ 
We fix two dg categories $\mathcal{A}$ and $\mathcal{B}$. We denote by $\mbox{rep($\mathcal{A},\mathcal{B}$)}$ the full triangulated subcategory of $\mathcal{D}(\mathcal{A}^{\tiny\mbox{op}}\otimes^{\mathbb{L}}{\mathcal{B}})$ spanned by the bimodules $M:(\mathcal{A}^{\tiny\mbox{op}}\otimes^{\mathbb{L}}{\mathcal{B}})^{\tiny\mbox{op}}\to\mathcal{C}^{\tiny\mbox{dg}}(k)$ such that $M(\mbox{-},x):\mathcal{B}^{\tiny\mbox{op}}\to\mathcal{C}^{\tiny\mbox{dg}}(k)$ is compact in $\mathcal{D}(\mathcal{B})$, for every $x\in\mathcal{A}$.\\
We have the following bijection
\begin{align*}\tag{$\dagger$}
\mbox{Iso[rep($\mathcal{A},\mathcal{B}$)]}&\to \mbox{Hom}_{\tiny\mbox{Hmo}}(\mathcal{A},\mathcal{B})\\
M&\mapsto \mbox{-}\otimes^{\mathbb{L}}_{\mathcal{A}} M,
\end{align*}
where Iso denotes the set of isomorphism classes of objects in $\mbox{rep($\mathcal{A},\mathcal{B}$)}$. Moreover, using the bijection ($\dagger$), we have an induced composition law in Hmo($k$).\\
Now we define the category $\mbox{Hmo($k$)}_0$ to be the category with the same objects as $\mbox{Hmo($k$)}$ and whose morphisms are given by $$\mbox{Hom}_{\tiny\mbox{Hmo($k$)}_0}(\mathcal{A},\mathcal{B}):=K_0(\mbox{rep}(\mathcal{A},\mathcal{B})).$$
Here $K_0(\mbox{rep($\mathcal{A},\mathcal{B}$)})$ denotes the Grothendieck group of the triangulated sub category $\mbox{rep($\mathcal{A},\mathcal{B}$)}$. As before, the composition law in $\mbox{Hmo($k$)}_0$ is the one induced by the bijection ($\dagger$).\\ 
We recall that the derived tensor product on $\mbox{Hmo($k$)}$ gives rise to a symmetric monoidal structure on $\mbox{Hmo($k$)}_0$.\\ 
We have a sequence of symmetric monoidal functors:
$$U:\mbox{dgcat}\to\mbox{Hmo($k$)}\to\mbox{Hmo($k$)}_0.$$
Finally, we denote by $\mbox{Hmo($k$)}_0^{\tiny\mbox{sp}}$ the full subcategory of the smooth and proper dg categories in $\mbox{Hmo($k$)}$.
We recall that if $\mathcal{A}$ is a smooth and proper dg category, then we have an equivalence of triangulated categories
\begin{align*}\tag{$\ddagger$}
\mbox{rep($\mathcal{A},\mathcal{B}$)}&\simeq\mathcal{D}_{\mbox{c}}(\mathcal{A}^{\tiny\mbox{op}}\otimes^{\mathbb{L}}{\mathcal{B}}),
\end{align*}
where $\mathcal{D}_{\mbox{c}}(\mathcal{A}^{\tiny\mbox{op}}\otimes^{\mathbb{L}}{\mathcal{B}})$ denotes the subcategory of compact objects in $\mathcal{D}(\mathcal{A}^{\tiny\mbox{op}}\otimes^{\mathbb{L}}{\mathcal{B}})$.\\
%When A is smooth and proper, we have an equivalence
%$rep(A, B)\simeq D_c(A^{op}\otimes^LB)$ of triangulated categories and so we obtain the following
Using the equivalence ($\ddagger$), we can describe the morphisms of $\mbox{Hmo($k$)}_0^{\tiny\mbox{sp}}$ as
$$\mbox{Hom}_{\tiny\mbox{Hmo($k$)}_0^{\tiny\mbox{sp}}}(\mathcal{A},\mathcal{B})\simeq K_0(\mathcal{D}_{\mbox{c}}(\mathcal{A}^{\tiny\mbox{op}}\otimes^{\mathbb{L}}{\mathcal{B}}))\simeq K_0(\mathcal{A}^{\tiny\mbox{op}}\otimes^{\mathbb{L}}{\mathcal{B}}).$$
Now we are ready to define the rigid symmetric monoidal category of noncommutative Chow motives.

\begin{defn}
We define the category of \emph{noncommutative Chow motives} to be the pseudoabelian envelope of $\mbox{Hom}_{\tiny\mbox{Hmo($k$)}_0^{\tiny\mbox{sp}}}$. We denote such a category by NChow($k$).
\end{defn}

\begin{rem}
We note that the functor $U$ extends naturally to NChow($k$).
\end{rem}

\begin{rem}
Let $X$ be an object in $\mbox{SmProj($k$)}$. If $\mbox{perf}(X)=\langle\mathcal{T}_1,...,\mathcal{T}_l\rangle$ then $U(\mbox{perf}_{\tiny\mbox{dg}}(X))=\mathcal{T}_1^{\tiny\mbox{dg}}\oplus...\oplus\mathcal{T}_l^{\tiny\mbox{dg}}$, where $\mathcal{T}_i^{\tiny\mbox{dg}}$ is the dg enhancement induced from $\mbox{perf}_{\tiny\mbox{dg}}(X)$. A proof of this result is in \cite{BeTa}.
\end{rem}

\begin{rem}
Given a commutative ring $F$, we can define the category $\mbox{NChow($k$)}_F$ taking the $F$-linearization $K_0(\mathcal{A}^{\tiny\mbox{op}}\otimes^{\mathbb{L}}{\mathcal{B}})_F$.
\end{rem}

\subsection{Voevodsky conjecture in the noncommutative case}
Let $\mathcal{A}$ be a smooth and proper dg category. We denote by $K_0(\mathcal{A})$ the Grothendieck group $K_0(\mathcal{D}_{\mbox{c}}(\mathcal{A}))$. In analogy to algebraic cycles we can define some equivalence relations on $K_0(\mathcal{A})$. We give two examples.

\begin{exmp}[$\otimes$-nilpotence equivalence relation]
We say that an element [$M$] in $K_0(\mathcal{A})$ is $\otimes$-nilpotent if there exists a positive integer $n$ such that $[M\times n]=0$ in the Grothendieck group $K_0(\mathcal{A}^{\otimes n})$. Given [$M$] and [$N$] in $K_0(\mathcal{A})$ we say that [$M$] and [$N$] are $\otimes$-nilpotent equivalent (shortly $\mbox{[$M$]$\sim_{\otimes_{\tiny\mbox{nil}}}$[$N$]}$) if $\mbox{[$M$]-[$N$]}$ is $\otimes$-nilpotent.
\end{exmp}
We have a bilinear form $\chi(-,-)$ on $K_0(\mathcal{A})$ defined as $$(M, N)\to\sum_i(-1)^i\mbox{dim Hom}_{\mathcal{D}_{\mbox{c}}(\mathcal{A})}(M, N[i]).$$
The left and right kernels of $\chi(-,-)$ are the same.

\begin{exmp}[Numerical equivalence relation]
We say that an element [$M$] in $K_0(\mathcal{A})$ is numerically trivial if $\chi([M], [N])=0$ for all $[N]\in K_0(\mathcal{A})$. We say that [$M$] and [$N$] are numerically trivial equivalent (shortly $\mbox{[$M$]$\sim_{\otimes_{\tiny\mbox{num}}}$[$N$]}$) if $\mbox{[$M$]-[$N$]}$ is numerically trivial.
\end{exmp}

\begin{rem}
The equivalence relations defined above give rise to well defined equivalence relations on $K_0(\mathcal{A})_F$.
\end{rem}
In \cite{BMT} Bernardara, Marcolli and Tabuada conjectured the following statement:

\begin{conj}[$\mbox{V}_{\tiny\mbox{nc}}$]
Let $\mathcal{A}$ be a smooth proper dg category. Then $\mbox{K}_0(\mathcal{A})/_{\sim\otimes_{\tiny\mbox{nil}}}$ is equal to $\mbox{K}_0(\mathcal{A})/_{\sim\otimes_{\tiny\mbox{num}}}$.
\end{conj}

\subsection{Pure motives vs noncommutative motives}
We recall that, given a smooth projective $k$-scheme, we have a relation between the category of Chow motives and the category of noncommutative Chow motives. In particular, we have the following commutative diagram
\[
\xymatrix{
\mbox{SmProj$(k)$}^{\tiny\mbox{op}}\ar[d]_{\mathfrak{h}}\ar[rr]^{\mbox{perf}_{\tiny\mbox{dg}}}&&\mbox{dgcat$(k)$}\ar[d]^U\\
\mbox{Chow$(k)$}\ar[rr]_{\Phi\cdot\eta}&&\mbox{NChow}(k)
}
\]
where $\Phi$ is fully faithful and $\eta$ is the functor from $\mbox{Chow$(k)$}$ to the orbit category $\mbox{Chow$(k)$}{/_{-\otimes\mathds{1}(1)}}$. Moreover, we have the following result which relates the Voevodsky's nilpotence conjecture and noncommutative motives:

\begin{namedthm}[{BMT}]\label{BMT1}
Let $X$ be a smooth projective $k$-scheme. The conjecture $\mbox{V}(X)$ holds if and only if the conjecture \emph{{$\mbox{V}_{\tiny\mbox{nc}}(\mbox{perf}_{\tiny\mbox{dg}}(X))$}} holds.
\end{namedthm}

\section{Kuznetsov category and GM category}

In this section we recall some facts about the decomposition of the derived category of a cubic fourfold $X$ (resp.\ of a GM variety). In particular, we remark some properties about the Kuznetsov category (resp.\ the GM category) $\mathcal{A}_X$ associated to $X$. Then we prove Voevodsky's nilpotence conjecture for the Kuznetsov category of a cubic fourfold and for the GM category of an ordinary generic GM fourfold.

\subsection{Kuznetsov category}
Let $X$ be a cubic fourfold. The derived category of perfect complexes $\mbox{perf}(X)$ admits a semiorthogonal decomposition given by 
\begin{align*}\tag{$\star$}
\mbox{perf}(X)=\langle\mathcal{A}_X,\mathcal{O}_X,\mathcal{O}_X(H),\mathcal{O}_X(2H)\rangle,
\end{align*}
where $H$ is a hyperplane section and $\mathcal{A}_X$ is defined as:
\begin{align*}
\mathcal{A}_X&=\langle\mathcal{O}_X,\mathcal{O}_X(H),\mathcal{O}_X(2H)\rangle^{\perp}\\
&=\mathcal{f}E\in\mbox{perf}(X)\mbox{ s.t. $\mathbb{R}\mbox{Hom}_{\tiny\mbox{perf}(X)}(\mathcal{O}_X(i),E)=0$\mbox{ for $i=0,1,2$}}\mathcal{g}
\end{align*}
We call $\mathcal{A}_X$ the Kuznetsov category. 

We recall that the triangulated subcategory $\mathcal{A}_X$ is a Calabi-Yau category of dimension $2$; indeed, the Serre functor is equal to the shift $\mbox{-}[2]$, i.e. for every pair of objects $F,E$ we have
$$\mathbb{R}\mbox{Hom}_{\mathcal{A}_X}(E,F)^*\simeq\mathbb{R}\mbox{Hom}_{\mathcal{A}_X}(F,E)[2].$$
Moreover, $\mathcal{A}_X$ has the same sized Hochschild (co)homology of the derived category of a K3 surface.\ Thus, the Kuznetsov category is a noncommutative K3 surface in the sense of Kontsevich (see \cite{Kuz}, \cite{Kuz3}, Corollary 4.3 and \cite{Kuz2}, Proposition 4.1).

\begin{rem}
We recall that if $X$ is a cubic fourfold containing a plane, we can prove $\mbox{V}$-conjecture via noncommutative motives.
In fact, if $X$ contains a plane, we have that $\mathcal{A}_X$ is equivalent to $\mathcal{D}^b(S,\mathcal{B})$, where $S$ is a K3 surface, $\mathcal{B}$ is a sheaf of Azumaya algebras on $S$ and $\mathcal{D}^b(S,\mathcal{B})$ is the derived category of coherent $\mathcal{B}$-modules on $S$ (see \cite{Kuz}, Theorem 4.3). Then by \cite{TaVdB} we have the following decomposition in $\mbox{NChow$(k)$}$: 
$$U(\mbox{perf}_{\tiny\mbox{dg}}(X)) \simeq U(\mbox{perf}_{\tiny\mbox{dg}}(S)) \oplus U(\mathbb{C})\oplus U(\mathbb{C})\oplus U(\mathbb{C}).$$ Since $\mbox{V}(S)$ holds, we conclude that also $V(X)$ holds, as we claimed.
\end{rem}

\subsection{GM category}
Let $X$ be a GM $n$-fold; in \cite{KuPe}, Proposition 4.2, they proved that its derived category of perfect complexes has a semiorthogonal decomposition of the form
\begin{equation*}\tag{$\ast$}
\mbox{perf}(X)=\langle\mathcal{A}_X,\mathcal{O}_X,\mathscr{U}^{\vee}_X,\mathcal{O}_X(H),\mathscr{U}^{\vee}_X(H),...,\mathcal{O}_X((n-3)H),\mathscr{U}^{\vee}_X((n-3)H)\rangle,
\end{equation*}
where $\mathscr{U}^{\vee}_X$ is the dual of the Gushel bundle previously defined and $\mathcal{A}_X$ is defined as:
\begin{align*}
\mathcal{A}_X&=\langle\mathcal{O}_X,\mathscr{U}^{\vee}_X,\mathcal{O}_X(H),\mathscr{U}^{\vee}_X(H),...,\mathcal{O}_X((n-3)H),\mathscr{U}^{\vee}_X((n-3)H)\rangle^{\perp}.
\end{align*}
We call $\mathcal{A}_X$ the GM category of $X$.

Assume that $X$ is a GM fourfold. Then, the Serre functor on $\mathcal{A}_X$ is the shift by two and the Hochschild cohomology of $\mathcal{A}_X$ is isomorphic to that of a K3 surface. In other words, the GM category of a GM fourfold is a noncommutative K3 surface in the sense of Kontsevich (see \cite{KuPe}, Proposition 5.18).

\subsection{Proof of conjecture $V_{\text{nc}}$ for the Kuznetsov category and the GM category of a generic GM fourfold}
\begin{B}[B]\label{B}
Let $X$ be a cubic fourfold or an ordinary generic GM fourfold. Then \emph{$V_{\tiny\mbox{nc}}(U(\mathcal{A}_X^{\tiny\mbox{dg}}))$} holds, where $\mathcal{A}_X^{\tiny\mbox{dg}}$ is the dg enhancement of $\mathcal{A}_X$ induced from \emph{$\mbox{perf}_{\tiny\mbox{dg}}(X)$}.
\end{B}

\begin{proof}
Let $X$ be a cubic fourfold. Using the decomposition $\star$, we have that the dg enhancement of the triangulated category $\mbox{perf}(X)$ admits the following decomposition in NChow($k$):
$$U(\mbox{perf}_{\tiny\mbox{dg}}(X))=U(\mathcal{A}_X^{\tiny\mbox{dg}}) \oplus U(\mathbb{C})\oplus U(\mathbb{C})\oplus U(\mathbb{C}).$$ Hence, the result is a straightforward consequence of \textbf{Theorem} (A).

The proof in the case of an ordinary generic GM fourfold $X$ is analogous, applying the decomposition $(\ast)$ and \textbf{Theorem} (A).
\end{proof}

\begin{rem}
We point out that \textbf{Theorem} (B) holds for every cubic fourfold $X$ even if it does not contain a plane.
\end{rem}

\section{Voevodsky's nilpotence conjecture for GM fourfolds containing surfaces}

In this section we will prove Voevodsky's nilpotence conjecture for generic GM fourfolds containing a $\tau$-plane and for ordinary GM fourfolds containing a quintic del Pezzo surface. The proof of the conjecture in the first case provide an application of \textbf{Theorem} (B).

%We recall some facts about the decomposition of the derived category of a GM varieties.
%Let $X$ be a GM $n$-fold, Kuznetsov and Perry \cite{KuPe} proved that its derived category of perfect complexes has the following semi-orthogonal decomposition:
%$$\mbox{perf}(X)=\langle\mathcal{A}_X,\mathcal{O}_X,\mathscr{U}^{\vee}_X,\mathcal{O}_X(H),\mathscr{U}^{\vee}_X(H),...,\mathcal{O}_X((n-3)H),\mathscr{U}^{\vee}_X((n-3)H)\rangle.$$
%Where $\mathscr{U}^{\vee}_X$ is the dual of the Gushel bundle previously defined, and $\mathcal{A}_X$ is defined as:
%\begin{align*}
%\mathcal{A}_X&=\langle\mathcal{O}_X,\mathscr{U}^{\vee}_X,\mathcal{O}_X(H),\mathscr{U}^{\vee}_X(H),...,\mathcal{O}_X((n-3)H),\mathscr{U}^{\vee}_X((n-3)H)\rangle^{\perp}
%\end{align*}
Let $X$ be a GM fourfold containing a $\tau$-plane $P$, i.e.\ a plane $P$ of the form $\G(2,V_3)$ for some $3$-dimensional subvector space $V_3$ of $V_5$. In \cite{KuPe}, Lemma 7.8, they proved that there exists a cubic fourfold $X'$ containing a smooth cubic surface scroll $T$ such that the blow-up of $X$ in $P$ is identified to the blow-up of $X'$ in $T$. More precisely, if $p: \tilde{X} \rightarrow X$ is the blow-up of $X$ along $P$ and $q$ is the regular map induced by the linear projection from $P$, then the diagram
\begin{equation}
\label{DC3}
\xymatrix{
&\ar[dl]_{p}\tilde{X}\ar[dr]^q&\\
X&& X'
}
\end{equation}
commutes and $q$ is identified with the blow-up of $X'$ along $T$. Moreover, they showed that if the GM fourfold $X$ does not contain a plane of the form $\p(V_1 \wedge V_4)$ for some subvectorspaces satisfying $V_1 \subset V_3 \subset V_4 \subset V_5$, then the cubic fourfold $X'$ is smooth.  We point out that this construction had already been described in \cite{DIM}. 

They also observed that a generic GM fourfold containing a $\tau$-plane does not contain a plane of the form $\p(V_1 \wedge V_4)$ as above; hence, the associated cubic fourfold $X'$ obtained with this geometric construction is smooth. In this case, they proved that there exists an equivalence of Fourier-Mukai type
\begin{equation}
\label{G}
\phi:\mathcal{A}_X\simeq\mathcal{A}_{X'}
\end{equation}
between the GM category of $X$ and the Kuznetsov category of $X'$ (see \cite{KuPe}, Theorem 1.3).

Using this construction and \textbf{Theorem} (B), we can prove the Voevodsky's nilpotence conjecture for this class of GM fourfolds.  
 
%\begin{rem}
%\label{G}
%In particular, if $X$ is a $GM$ fourfold containing a plane $P$ of type Gr(2,3), we know that there exist a cubic fourfold $X'$ and an equivalence of Fourier-Mukai type $\phi:\mathcal{A}_X\simeq\mathcal{A}_{X'}$. Where $\mathcal{A}_{X'}$ is the Kuznetsov category defined before.
%\end{rem}
\begin{namedthm}[C]
\label{T}
Let $X$ be a generic GM fourfold containing a plane $P$ of type $\emph{Gr}(2,3)$. Then \emph{$V_{\tiny\mbox{nc}}(\mbox{perf}_{\tiny\mbox{dg}}(X))$} holds.
\end{namedthm}

\begin{proof}
The derived category of perfect complexes of $X$ has the following decomposition:
$$\mbox{perf}(X)=\langle\mathcal{A}_X,\mathcal{O}_X,\mathscr{U}^{\vee}_X,\mathcal{O}_X(H),\mathscr{U}^{\vee}_X(H)\rangle.$$ 
Since the functor $\phi$ defined in \eqref{G} is of Fourier-Mukai type, we know that $\phi$ has a dg lift, thanks to the works of \cite{LuS}, \cite{Sch} and \cite{Toe}. Then the proof is a consequence of \textbf{Theorem} (B).
\end{proof}

\begin{rem}
Alternatively, we can prove \textbf{Theorem} (C) by observing that the isomorphism of triangulated categories $\mathcal{A}_X\simeq \mathcal{A}_{X'}$ is induced by diagram \eqref{DC3}. Then, conjecture $V_{\tiny\mbox{nc}}(\mbox{perf}_{\tiny\mbox{dg}}(X))$ follows from Subsection \ref{subsec:b} and \textbf{Theorem} (A).   
\end{rem}

In a similar fashion, we can prove conjecture $\text{V}_{\text{nc}}$ for the category of perfect complexes of ordinary GM fourfolds containing a quintic del Pezzo surface.

\begin{namedthm}[D]
Let $X$ be an ordinary GM fourfold containing a quintic del Pezzo surface. Then \emph{$V_{\tiny\mbox{nc}}(\mbox{perf}_{\tiny\mbox{dg}}(X))$} holds.
\end{namedthm}

\begin{proof}
By \cite{KuPe}, Theorem 1.2 we have that there exist a K3 surface $Y$ and an equivalence $\psi: \mathcal{A}_X\simeq \mathcal{D}^b(Y)$ of Fourier-Mukai type. Since $\psi$ has a dg lift and conjecture V holds for $Y$, the proof follows from Theorem \ref{BMT1}. 
%the derived category of perfect complexes of $X$ has the following decomposition:
%$$\mbox{perf}(X)=\langle D(Y),\mathcal{O}_X,\mathscr{U}^{\vee}_X,\mathcal{O}_X(H),\mathscr{U}^{\vee}_X(H)\rangle.$$ 
%Where $Y$ is a K3 surface. Then by subsection \ref{BMT1} follows the proof.
\end{proof}

\begin{cor}
Let $X$ be a generic GM fourfold containing a plane $P$ of type $\emph{Gr}(2,3)$ or an ordinary GM fourfold containing a quintic del Pezzo surface.\ Then \emph{$V(X)$}\ %\emph{$K(X)$} and \emph{$S(X)$} 
holds.
\end{cor}

\begin{proof}
The proof is a consequence of Theorem \ref{BMT1}.% and Theorem \ref{CC}.
\end{proof}

\end{document}